\documentclass[11pt]{amsart}

\usepackage{amssymb, amscd}
\usepackage{epsfig, mathtools}
\usepackage[vmargin=1in, hmargin=1.3in]{geometry}
\usepackage[font=small,format=plain,labelfont=bf,up,textfont=it,up]{caption}
\usepackage{microtype}
\usepackage[shortlabels]{enumitem}
\usepackage[backref=page, bookmarks, bookmarksdepth=2, colorlinks=true, linkcolor=blue, citecolor=blue, urlcolor=blue]{hyperref}
\usepackage{etoolbox}
\usepackage{tikz-cd}
\apptocmd{\thebibliography}{\raggedright}{}{}

\makeatletter
\patchcmd{\@maketitle}{\global\topskip42\p@\relax}
  {\global\topskip42\p@\relax \vspace*{-38pt}}
  {}{}
\makeatother

\setenumerate[0]{label=(\alph*)}

\renewcommand*{\backref}[1]{}
\renewcommand*{\backrefalt}[4]{%
    \ifcase #1 (Not cited.)%
    \or        (Cited on page~#2.)%
    \else      (Cited on pages~#2.)%
    \fi}

\newcommand{\arxiv}[1]{\href{http://arxiv.org/abs/#1}{{\tt arXiv:#1}}}

\numberwithin{equation}{section}

\theoremstyle{plain}
\newtheorem{theorem}{Theorem}[section]
\newtheorem{maintheorem}{Theorem}

\newtheorem{maintheoremprime}{Theorem}
\newtheorem{proposition}[theorem]{Proposition}
\newtheorem{lemma}[theorem]{Lemma}

\newtheorem{conjecture}[theorem]{Conjecture}
\newtheorem{question}[theorem]{Question}

\theoremstyle{definition}
\newtheorem{asm}[theorem]{Assumption}

\newtheorem{defn}[theorem]{Definition}

\newtheorem{notn}[theorem]{Notation}

\theoremstyle{remark}
\newtheorem{rmk}[theorem]{Remark}
\newenvironment{remark}[1][]{\begin{rmk}[#1] \pushQED{\qed}}{\popQED \end{rmk}}
\newtheorem{eg}[theorem]{Example}

\DeclareMathOperator{\Hom}{Hom}
\DeclareMathOperator{\End}{End}

\DeclareMathOperator{\Sp}{Sp}

\newcommand\Z{\ensuremath{\mathbb{Z}}}
\newcommand\Q{\ensuremath{\mathbb{Q}}}

\DeclareMathOperator{\HH}{H}

\DeclareMathOperator{\lcm}{lcm}

\newcommand\Span[1]{\ensuremath{\langle #1 \rangle}}
\newcommand\Set[2]{\ensuremath{\left\{\text{#1 $|$ #2}\right\}}}
\newcommand\SpanSet[2]{\ensuremath{\left\langle \text{#1 $|$ #2} \right\rangle}}

\newcommand\cM{\ensuremath{\mathcal{M}}}

\newcommand\cP{\ensuremath{\mathcal{P}}}

\newcommand\cU{\ensuremath{\mathcal{U}}}

\newcommand\bD{\ensuremath{\mathbf{D}}}

\newcommand\bG{\ensuremath{\mathbf{G}}}

\newcommand\bc{\ensuremath{\mathbf{c}}}

\newcommand\bu{\ensuremath{\mathbf{u}}}
\newcommand\bv{\ensuremath{\mathbf{v}}}

\newcommand\bx{\ensuremath{\mathbf{x}}}

\newcommand\tS{\ensuremath{\widetilde{S}}}

\newcommand\tphi{\ensuremath{\widetilde{\phi}}}
\newcommand\talpha{\ensuremath{\widetilde{\alpha}}}
\newcommand\tbeta{\ensuremath{\widetilde{\beta}}}
\newcommand\tgamma{\ensuremath{\widetilde{\gamma}}}
\newcommand\tdelta{\ensuremath{\widetilde{\delta}}}

\newcommand\ttau{\ensuremath{\widetilde{\tau}}}
\newcommand\tSigma{\ensuremath{\widetilde{\Sigma}}}

\newcommand\tcU{\ensuremath{\widetilde{\cU}}}

\newcommand\oP{\ensuremath{\overline{P}}}

\newcommand\ou{\ensuremath{\overline{u}}}
\newcommand\ov{\ensuremath{\overline{v}}}

\newcommand\hell{\ensuremath{\widehat{\ell}}}

\newcommand\obu{\ensuremath{\overline{\bu}}}
\newcommand\obv{\ensuremath{\overline{\bv}}}

\newcommand\hobu{\ensuremath{\widehat{\obu}}}
\newcommand\hobv{\ensuremath{\widehat{\obv}}}

\DeclareMathOperator{\Mod}{Mod}

\DeclareMathOperator{\scc}{scc}
\DeclareMathOperator{\pant}{pant}

\newcommand\HHScc{\ensuremath{\HH^{\scc}}}
\newcommand\HHPants{\ensuremath{\HH^{\pant}}}

\title[Generating the homology of covers of surfaces]{Generating the homology of covers of surfaces}

\subjclass[2020]{Primary 57K20; Secondary 57M10, 57M12}

\author{Marco Boggi}
\address{UFF - Instituto de Matem\'atica e Estat\'{\i}stica -
Niter\'oi - RJ 24210-200; Brazil}
\email{marco.boggi@gmail.com}

\author{Andrew Putman}
\address{Dept of Mathematics; Univ of Notre Dame; 255 Hurley Hall; Notre Dame, IN 46556; USA}
\email{andyp@nd.edu}
\thanks{AP was supported in part by NSF grant DMS-1811322.}

\author{Nick Salter}
\address{Dept of Mathematics; Univ of Notre Dame; 255 Hurley Hall; Notre Dame, IN 46556; USA}
\email{nsalter@nd.edu}
\thanks{NS was supported in part by NSF grant DMS-2153879.} 

\date{January 22, 2024}

\begin{document}

\newpage

\begin{abstract}
Putman and Wieland conjectured that if $\tSigma \rightarrow \Sigma$ is a finite branched cover between closed
oriented surfaces of sufficiently high genus, then the orbits of all nonzero elements of $\HH_1(\tSigma;\Q)$ under the action of lifts
to $\tSigma$ of mapping classes on $\Sigma$ are infinite.  We prove that this holds if $\HH_1(\tSigma;\Q)$ is
generated by the homology classes of lifts of simple closed curves on $\Sigma$.  We also prove that the subspace
of $\HH_1(\tSigma;\Q)$ spanned by such lifts is a symplectic subspace.  Finally, simple closed curves lie
on subsurfaces homeomorphic to $2$-holed spheres, and we prove that $\HH_1(\tSigma;\Q)$ is generated
by the homology classes of lifts of loops on $\Sigma$ lying on subsurfaces homeomorphic to $3$-holed
spheres.
\end{abstract}

\maketitle
\thispagestyle{empty}

\section{Introduction}

Let $\pi\colon \tSigma \rightarrow \Sigma$ be a finite branched cover between closed oriented surfaces.
The homology of $\tSigma$ encodes subtle information about the mapping class group of $\Sigma$, and
over the last decade has been intensely studied \cite{GrunewaldEtAl, Klukowski, KoberdaSantharoubane, LandesmanLitt1, LandesmanLitt2, LandesmanLitt3, LooijengaPrym, MalesteinPutman, Markovic, MarkovicTosic, PutmanWieland}.  
Much of this is motivated by a conjecture of Putman--Wieland \cite{PutmanWieland} we discuss
below.  In this note, we prove this conjecture for covers $\tSigma$ such that $\HH_1(\tSigma;\Q)$ is
generated by certain simple elements, and also prove that in general $\HH_1(\tSigma;\Q)$ is generated
by slightly more complicated elements.

\subsection{Putman--Wieland conjecture}

Mark $\Sigma$ at each branch point of the branched cover $\pi\colon \tSigma \rightarrow \Sigma$. 
Let $\Mod(\Sigma)$ be the pure mapping class group of $\Sigma$, i.e., the group
of isotopy classes of orientation-preserving homeomorphisms of $\Sigma$ that fix each marked point.  There
is a finite-index subgroup $\Mod(\Sigma,\tSigma)$ of $\Mod(\Sigma)$ that can be lifted to $\tSigma$ to give a well-defined
action of $\Mod(\Sigma,\tSigma)$ on $\HH_1(\tSigma;\Q)$.  Putman--Wieland \cite{PutmanWieland} made the following conjecture.

\begin{conjecture}[\cite{PutmanWieland}]
\label{conjecture:putmanwieland}
Let the notation be as above, and assume that the genus of $\Sigma$ is sufficiently large.\footnote{In \cite{PutmanWieland} they
further conjecture that this holds if the genus is at least $2$, but counterexamples in genus $2$ were found
by Markovi\'{c} \cite{Markovic}.}  Consider some nonzero $\vec{v} \in \HH_1(\tSigma;\Q)$.  Then 
the $\Mod(\Sigma,\tSigma)$-orbit of $\vec{v}$ is infinite.
\end{conjecture}

The main theorem of \cite{PutmanWieland} says that this holds if and only if the virtual first Betti number
of the mapping class group is $0$ when the genus is sufficiently large, which is a well-known conjecture
of Ivanov \cite{IvanovConjecture}.

\subsection{Simple closed curve homology}

To prove Conjecture \ref{conjecture:putmanwieland}, it is natural to try to find generators
for $\HH_1(\tSigma;\Q)$.
A first idea is that $\HH_1(\tSigma;\Q)$ might be generated by lifts of simple closed curves.  Define
the {\em simple closed curve homology} of $\tSigma$, denoted $\HHScc_1(\tSigma;\Q)$, to be the subspace
of $\HH_1(\tSigma;\Q)$ spanned by the homology classes of loops $\tgamma$ on $\tSigma$ that avoid the branch
points and project to simple closed curves $\gamma$ on $\Sigma$.  The restriction of the branched covering map
$\pi\colon \tSigma \rightarrow \Sigma$ to $\tgamma$ is thus a possibly nontrivial cover
$\tgamma \rightarrow \gamma$.  

Unfortunately, this need not be all of $\HH_1(\tSigma;\Q)$.  For all closed surfaces $\Sigma$ with $\pi_1(\Sigma)$ nonabelian,
Malestein--Putman \cite[Theorem B]{MalesteinPutman} constructed finite branched
covers $\tSigma \rightarrow \Sigma$ such that $\HHScc_1(\tSigma;\Q) \neq \HH_1(\tSigma;\Q)$. 
More recently, Klukowski \cite{Klukowski} constructed unbranched covers with this property.\footnote{Earlier
Koberda--Santharoubane \cite{KoberdaSantharoubane} constructed unbranched covers of closed surfaces
with $\HHScc_1(\tSigma;\Z) \neq \HH_1(\tSigma;\Z)$.  This is weaker since it is possible that in their examples
$\HHScc_1(\tSigma;\Z)$ is a finite-index subgroup of $\HH_1(\tSigma;\Z)$.}
Our first theorem is that Conjecture \ref{conjecture:putmanwieland} does hold if
if $\HHScc_1(\tSigma;\Q) = \HH_1(\tSigma;\Q)$. 

\begin{maintheorem}
\label{maintheorem:putmanwieland}
Let $\pi\colon \tSigma \rightarrow \Sigma$ be a finite branched cover between closed
oriented surfaces.  Consider some nonzero $\vec{v} \in \HHScc_1(\tSigma;\Q)$.  Then
the $\Mod(\Sigma,\tSigma)$-orbit of $\vec{v}$ is infinite.  In particular, if
$\HHScc_1(\tSigma;\Q) = \HH_1(\tSigma;\Q)$ then Conjecture \ref{conjecture:putmanwieland} holds
for $\pi\colon \tSigma \rightarrow \Sigma$.
\end{maintheorem}

This suggests that the examples
from \cite{MalesteinPutman} and \cite{Klukowski} might be good places to look
for counterexamples
to Conjecture \ref{conjecture:putmanwieland}.

\subsection{Relationship to previous work}

Conjecture \ref{conjecture:putmanwieland} has been proved in a variety of cases; see,\footnote{Not all of
these papers explicitly prove cases of Conjecture \ref{conjecture:putmanwieland}, but it can be deduced from their
results in the cases they cover.} e.g., 
\cite{GrunewaldEtAl, LandesmanLitt1, LandesmanLitt2, LandesmanLitt3, LooijengaPrym}.  However, we know
very little about when $\HHScc_1(\tSigma;\Q) = \HH_1(\tSigma;\Q)$.  The only general result
we are aware of is that this holds when $\tSigma \rightarrow \Sigma$ is a finite unbranched abelian cover.  This
is implicit in work of Looijenga \cite{LooijengaPrym}, and we provide a self-contained
proof in Proposition \ref{proposition:sccabelian} below.\footnote{Proposition
\ref{proposition:sccabelian} is stronger than this: it shows that
$\HH_1(\tSigma;\Q)$ is spanned by lifts of {\em nonseparating} simple closed curves.}  Beyond this, it is unclear which known
cases of Conjecture \ref{conjecture:putmanwieland} follow from 
Theorem \ref{maintheorem:putmanwieland}.

\begin{remark}
It would be interesting to extend this to prove that $\HHScc_1(\tSigma;\Q) = \HH_1(\tSigma;\Q)$
for finite branched abelian covers.
\end{remark}

Next, let $\pi\colon \tSigma \rightarrow \Sigma$ be one of the examples from \cite{MalesteinPutman} 
or \cite{Klukowski} where $\HHScc_1(\tSigma;\Q) \neq \HH_1(\tSigma;\Q)$.  Below in Theorem \ref{maintheorem:sccsymplectic} we will prove
that $\HHScc_1(\tSigma;\Q)$ is a symplectic subspace of $\HH_1(\tSigma;\Q)$,  so
\[\HH_1(\tSigma;\Q) = \HHScc_1(\tSigma;\Q) \oplus \HHScc_1(\tSigma;\Q)^{\perp}.\]  
Theorem \ref{maintheorem:putmanwieland} says that the $\Mod(\Sigma,\tSigma)$-orbit
of all nonzero $\vec{v} \in \HHScc_1(\tSigma;\Q)$ is infinite.  It turns out
that there are also nonzero $\vec{v} \in \HHScc_1(\tSigma;\Q)^{\perp}$ whose
$\Mod(\Sigma,\tSigma)$-orbits are infinite.  Indeed, letting $D$ be the deck
group of $\pi\colon \tSigma \rightarrow \Sigma$, it follows from the constructions
in \cite{MalesteinPutman} and \cite{Klukowski} that some $D$-isotypic subspace $V$
of $\HH_1(\tSigma;\Q)$ lies in $\HHScc_1(\tSigma;\Q)^{\perp}$, and
Landesman--Litt \cite{LandesmanLitt2} proved that some nonzero $\vec{v} \in V$ has
an infinite $\Mod(\Sigma,\tSigma)$-orbit.

\subsection{Symplectic subspace}
\label{section:symplecticsubspace}

We next clarify 
the nature of the subspace $\HHScc_1(\tSigma;\Q)$ of $\HH_1(\tSigma;\Q)$.  By Poincar\'{e} duality, the algebraic
intersection form $\omega$ on $\HH_1(\tSigma;\Q)$ is a {\em symplectic form}, i.e., an alternating
bilinear form that induces an isomorphism between $\HH_1(\tSigma;\Q)$ and its dual.  A subspace
$V$ of $\HH_1(\tSigma;\Q)$ is a {\em symplectic subspace} if the restriction of $\omega$ to $V$
is a symplectic form.  We will prove the following:

\begin{maintheorem}
\label{maintheorem:sccsymplectic}
Let $\pi\colon \tSigma \rightarrow \Sigma$ be a finite branched cover between
closed oriented surfaces.
Then $\HHScc_1(\tSigma;\Q)$ is a symplectic subspace of $\HH_1(\tSigma;\Q)$.
\end{maintheorem}

In fact, we will prove something more general.  A {\em nontrivial simple closed curve} on $\Sigma$
is a simple closed curve $\gamma$ that avoids the marked points and does not bound a disk containing
at most one marked point.  We will always consider such curves up to isotopy.\footnote{These
are isotopies through nontrivial simple closed curves, so during the isotopies the curves
cannot pass through the marked points.}  The group $\Mod(\Sigma)$
acts on the set of nontrivial simple closed curves on $\Sigma$, and the orbits of
this action are the {\em topological types} of nontrivial simple closed curves.\footnote{By 
the change of coordinates principle from \cite[\S 1.3.2]{Primer}, the topological types are
determined by the marked surface with boundary one gets by cutting $\Sigma$ open along $\gamma$.
For instance, one topological type is the set of all nonseparating $\gamma$.}

If $\sigma$ is a set of topological types of nontrivial simple closed curves on $\Sigma$, then
denote by $\HH_1^{\sigma}(\tSigma;\Q)$ the subspace
of $\HH_1(\tSigma;\Q)$ spanned by the homology classes of loops $\tgamma$ on $\tSigma$ that avoid the branch
points and project to simple closed curves $\gamma$ on $\Sigma$ such that the topological type of
$\gamma$ lies in $\sigma$.  For instance, if $\sigma$ is the set of all topological types of 
nontrivial simple closed curves on $\Sigma$, then
\[\HH_1^{\sigma}(\Sigma;\Q) = \HHScc_1(\tSigma;\Q).\]
The following therefore generalizes Theorem \ref{maintheorem:sccsymplectic}:

\stepcounter{maintheoremprime}
\begin{maintheoremprime}
\label{maintheorem:sccsymplectic2}
Let $\pi\colon \tSigma \rightarrow \Sigma$ be a finite branched cover between closed oriented
surfaces and $\sigma$ be a set of topological types of nontrivial simple closed curves on $\Sigma$.  Then
$\HH_1^{\sigma}(\tSigma;\Q)$ is a symplectic subspace of $\HH_1(\tSigma;\Q)$.
\end{maintheoremprime}

We can also define $\HH_1^{\sigma}(\tSigma;\Z)$ and $\HH_1^{\scc}(\tSigma;\Z)$, and it is natural
to wonder whether Theorems \ref{maintheorem:sccsymplectic} and \ref{maintheorem:sccsymplectic2} hold
integrally.  For Theorem \ref{maintheorem:sccsymplectic2}, the answer is no in general:

\begin{maintheorem}
\label{maintheorem:nonsymplectic}
Let $\Sigma$ be a closed oriented surface of genus at least $2$ and let $\sigma$ be the set
of nonseparating simple closed curves on $\Sigma$.  Then
there exists a finite unbranched cover $\pi\colon \tSigma \rightarrow \Sigma$ such that
$\HH_1^{\sigma}(\tSigma;\Z)$ is not a symplectic subspace of $\HH_1(\tSigma;\Z)$.
\end{maintheorem}

Here $\HH_1^{\sigma}(\tSigma;\Z)$ is a free abelian group, and a symplectic form on a free abelian group $A$
is an alternating $\Z$-valued bilinear form on $A$ that identifies $A$ with its dual $A^{\ast} = \Hom(A,\Z)$.
Unfortunately, our proof of Theorem \ref{maintheorem:nonsymplectic} breaks down if we allow separating
curves, so we cannot answer the following question:

\begin{question}
\label{question:symplectic}
Let $\pi\colon \tSigma \rightarrow \Sigma$ be a finite branched cover between closed oriented
surfaces.  Is $\HH_1^{\scc}(\tSigma;\Z)$ a symplectic subspace of $\HH_1(\tSigma;\Z)$?
\end{question}

However, Theorem \ref{maintheorem:nonsymplectic} suggests that the answer to this should be ``no''.

\begin{remark}
An important ingredient in our proof of Theorem \ref{maintheorem:nonsymplectic} is a theorem of Irmer \cite{IrmerCovers}
giving certain finite abelian covers $\pi\colon \tSigma \rightarrow \Sigma$ for which
$\HH_1^{\sigma}(\tSigma;\Z)$ is a proper subgroup of $\HH_1(\tSigma;\Z)$ (see Theorem \ref{theorem:nonsymplectic}.(ii) below).  
To make this paper more self-contained, we also include a simplified proof of this theorem.
\end{remark}

\subsection{Pants homology}

Regular neighborhoods of simple closed curves on $\Sigma$ are homeomorphic to annuli, i.e., spheres
with two boundary components.  This suggests weakening the definition of simple closed curve homology
as follows.  

Recall that a {\em pair of pants} is a sphere with three holes.  Define the {\em pants
homology} of $\tSigma$, denoted $\HHPants_1(\tSigma;\Q)$, to be the subspace of
$\HH_1(\tSigma;\Q)$ spanned by the homology classes of loops $\tgamma$ on $\tSigma$ such that there
exists a subsurface $P \subset \Sigma$ homeomorphic to a pair of pants with $\pi(\tgamma) \subset P$.
Since every simple closed curve on $\Sigma$ is contained in some\footnote{We do not require the boundary
components of $P$ to be non-nullhomotopic curves on $\Sigma$, so this even holds if $\Sigma$ is a surface like a sphere that
does not contain pairs of pants $P$ whose boundary components are non-nullhomotopic.} such $P$, we have 
\[\HHScc_1(\tSigma;\Q) \subset \HHPants_1(\tSigma;\Q).\]
Our final main theorem is as follows.  It answers positively a question\footnote{Kent actually asked whether
$\HH_1(\tSigma;\Q)$ is generated by lifts of elements that do not fill $\Sigma$, which is much weaker
than lying in a pair of pants.} of Kent \cite{KentMO}.

\begin{maintheorem}
\label{maintheorem:pants}
Let $\pi\colon \tSigma \rightarrow \Sigma$ be a finite branched cover between closed
oriented surfaces.  Then
$\HHPants_1(\tSigma;\Q) = \HH_1(\tSigma;\Q)$.
\end{maintheorem}

\begin{remark}
This also holds for punctured surfaces of finite type, which
can be reduced to Theorem \ref{maintheorem:pants} as follows.  Let 
$\pi\colon \tSigma \rightarrow \Sigma$
be a finite branched cover between punctured surfaces of finite type.  Filling 
in the punctures yields a finite branched
cover between closed surfaces to which one can apply Theorem \ref{maintheorem:pants}.
To conclude, note that 
filling in the punctures has the effect of killing the homology
classes in $\HH_1(\tSigma;\Q)$ of loops around the punctures, which lie
in $\HHScc_1(\tSigma;\Q) \subset \HHPants_1(\tSigma;\Q)$.
\end{remark}

\begin{remark}
Theorem \ref{maintheorem:pants} might appear to contradict \cite[Theorem C]{MalesteinPutman} and
\cite[Corollary 1.1.3]{Klukowski}, which give examples of finite covers
$\pi\colon \tSigma \rightarrow \Sigma$
such that $\HH_1(\tSigma;\Q)$ is not spanned by the homology classes of loops $\tgamma$ such that
$\pi(\tgamma)$ is not in any given finite set of mapping class group\footnote{or even automorphism
group of free group} orbits of curves.  However, in the definition of $\HHPants_1(\tSigma;\Q)$ there is no
restriction on the number of self-intersections of the projections of the curves to the $P$,
so they do not fall into finitely many mapping class group orbits.
\end{remark}

The proof of Theorem \ref{maintheorem:pants} actually shows something stronger.  A {\em pants decomposition}
of $\Sigma$ is a collection $\cP = \{\delta_1,\ldots,\delta_n\}$ of disjoint simple closed curves on $\Sigma$ that avoid
the branch points such that each component of
$\Sigma \setminus \cup_{j=1}^n \delta_j$
is either a disk containing a single branch point or a pair of pants containing no branch points:\\
\centerline{\psfig{file=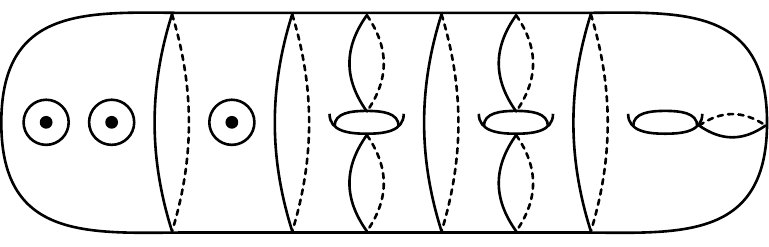,scale=1}}
We will prove the following, which implies Theorem \ref{maintheorem:pants}:

\stepcounter{maintheoremprime}
\begin{maintheoremprime}
\label{maintheorem:pants2}
Let $\pi\colon \tSigma \rightarrow \Sigma$ be a finite branched cover between closed
oriented surfaces
and $\cP$ be a pants decomposition of $\Sigma$.  Let $\sigma$ be the set of topological types of nontrivial curves
appearing in $\cP$.  Then $\HH_1(\tSigma;\Q)$ is spanned by $\HH_1^{\sigma}(\tSigma;\Q)$ and
the set of homology classes of cycles $\tgamma$ on $\tSigma$ such that $\pi(\tgamma)$ is disjoint 
from all curves in $\cP$.
\end{maintheoremprime}

We can also define $\HHPants_1(\tSigma;\Z)$, and pose the following question:
    
\begin{question}
\label{question:pants}
Let $\pi\colon \tSigma \rightarrow \Sigma$ be a finite branched cover between closed
oriented surfaces.
Is $\HHPants_1(\tSigma;\Z) = \HH_1(\tSigma;\Z)$?
\end{question}

Our proof of Theorems \ref{maintheorem:pants} and \ref{maintheorem:pants2} shows that Question \ref{question:pants} has a positive answer if Question \ref{question:symplectic} does.  Though we expect that Question \ref{question:symplectic} has a negative
answer, we do not know what answer to expect for Question \ref{question:pants}.

\subsection{Outline}
We prove Theorem \ref{maintheorem:putmanwieland} in \S \ref{section:putmanwieland}, Theorems
\ref{maintheorem:sccsymplectic} and \ref{maintheorem:sccsymplectic2} in \S \ref{section:sccsymplectic}, Theorems
\ref{maintheorem:pants} and \ref{maintheorem:pants2} in \S \ref{section:pants}, and Theorem \ref{maintheorem:nonsymplectic}
in \S \ref{section:nonsymplectic}.

\subsection{Acknowledgments}
We would like to thank Eduard Looijenga for his help, especially with Theorem \ref{maintheorem:sccsymplectic}.
We would also like to thank Aaron Landesman and Daniel Litt for helpful comments on a previous version
of this paper.

\section{Simple closed curve homology and Dehn twists}
\label{section:putmanwieland}

This section contains the proof of Theorem \ref{maintheorem:putmanwieland}.

\subsection{Notation}

Fix a closed oriented surface $\Sigma$ and a finite branched cover $\pi\colon \tSigma \rightarrow \Sigma$.
The surface $\tSigma$ is thus also a closed oriented surface.  Let $g \geq 0$ be its genus, let
\[H =\HH_1(\tSigma;\Q) \cong \Q^{2g},\]
and let $\omega$ be the algebraic intersection form on $H$.
By Poincar\'{e} duality, $\omega$ is a symplectic form, i.e., an alternating form that 
induces an isomorphism between $H$ and its dual.  The symplectic group $\Sp(H,\omega) \cong \Sp_{2g}(\Q)$ acts on
$H$.

\subsection{Lifting Dehn twists}
\label{section:lifttwists}

Recall from \S \ref{section:symplecticsubspace} that a nontrivial simple closed curve
on $\Sigma$ is a simple closed curve that avoids the 
marked points and does not bound a disk containing at most one marked point.
Consider a nontrivial simple closed curve $\gamma$ on $\Sigma$.  The preimage
$\pi^{-1}(\gamma)$ is a disjoint union of simple closed curves.  Enumerate them as
\[\pi^{-1}(\gamma) = \tgamma_1 \sqcup \cdots \sqcup \tgamma_k.\]
For each $1 \leq j \leq k$, the map
\[\pi|_{\tgamma_j}\colon \tgamma_j \rightarrow \gamma\]
is a finite unbranched cover.  Let $d_j$ be its degree.  Set
\begin{equation}
\label{eqn:defined}
d(\gamma) = \lcm(d_1,\ldots,d_k) \quad \text{and} \quad e_j = d(\gamma) / d_j \quad \text{for $1 \leq j \leq k$}.
\end{equation}
If $T_{\gamma}$ and $T_{\tgamma_j}$ denote the Dehn twists about $\gamma$ and $\tgamma_j$, then
$T_{\gamma}^{d(\gamma)}$ lifts to the product
\[T_{\tgamma_1}^{e_1} \cdots T_{\tgamma_k}^{e_k}.\]
Let $\ttau_{\gamma}$ be the image of this product of powers of Dehn twists in
$\Sp(H,\omega) \cong \Sp_{2g}(\Q)$.
The element $\ttau_{\gamma}$ acts on $H$ as follows:
\[\ttau_{\gamma}(h) = h + \sum_{j=1}^k e_j \omega(h,[\tgamma_j]) \cdot [\tgamma_j] \quad \text{for $h \in H$}.\]
For a set $\sigma$ of topological types of nontrivial simple closed curves on $\Sigma$,
define $D_{\sigma}$ to be the subgroup of $\Sp(H,\omega)$ generated by the set
of all $\ttau_{\gamma}$ as $\gamma$ ranges over nontrivial simple closed curves on $\Sigma$ 
whose topological type lies in $\sigma$.

\subsection{Fixed set of lifted twists}

As in \S \ref{section:symplecticsubspace}, let $\sigma$ be a set of topological
types of nontrivial simple closed curves on $\Sigma$ and define
\[H^{\sigma} = \HH_1^{\sigma}(\tSigma;\Q) \subset \HH_1(\tSigma;\Q) = H.\]
The following lemma will be fundamental to our paper:

\begin{lemma}
\label{lemma:twistfixed}
Let the notation be as above.  Then\footnote{Here the superscript indicates that
we are taking invariants: $H^{D_{\sigma}} = \Set{$h \in H$}{$d \cdot h = h$ for all $d \in D_{\sigma}$}$.} $H^{D_{\sigma}}$
equals the orthogonal complement $(H^{\sigma})^{\perp}$ of $H^{\sigma}$ with respect to $\omega$.
\end{lemma}
\begin{proof}
Let $\gamma$ be a nontrivial simple closed curve on $\Sigma$ whose topological type
lies in $\sigma$.  As we did in
\S \ref{section:lifttwists} above, write $\pi^{-1}(\gamma)$ as a disjoint union of simple
closed curves on $\tSigma$:
\[\pi^{-1}(\gamma) = \tgamma_1 \sqcup \cdots \sqcup \tgamma_k.\]
As in that section, there are positive integers $e_1,\ldots,e_k$ such that 
the generator $\ttau_{\gamma} \in D_{\sigma}$ acts on $H$ as follows:
\[\ttau_{\gamma}(h) = h + \sum_{j=1}^k e_j \omega(h,[\tgamma_j]) \cdot [\tgamma_j] \quad \text{for $h \in H$}.\]
Each $[\tgamma_j]$ lies in $H^{\sigma}$, so it is immediate from this formula that $(H^{\sigma})^{\perp} \subset H^{D_{\sigma}}$.
For the other inclusion, consider some $h_0 \in H^{D_{\sigma}}$.  We then know that $\ttau_{\gamma}(h_0) = h_0$, so from
the above
\[\sum_{j=1}^k e_j \omega(h_0,[\tgamma_j]) \cdot [\tgamma_j] = 0.\]
Taking the algebraic intersection with $h_0$, we deduce that
\[\sum_{j=1}^k e_j \omega(h_0,[\tgamma_j])^2 = 0.\]
Since $e_j \geq 1$ for all $1 \leq j \leq k$, this implies that $\omega(h_0,[\tgamma_j])=0$ for all $1 \leq j \leq k$.
This holds for all choices of $\gamma$ and all components of the preimage $\pi^{-1}(\gamma)$.  These
generate $H^{\sigma}$, so we conclude that $h_0 \in (H^{\sigma})^{\perp}$, as desired.
\end{proof}

\subsection{Putman--Wieland conjecture}

We now prove Theorem \ref{maintheorem:putmanwieland}.

\begin{proof}[Proof of Theorem \ref{maintheorem:putmanwieland}]
We start by recalling the statement.  Let
$\pi\colon \tSigma \rightarrow \Sigma$ be a finite branched cover between closed oriented surfaces.
Let $\Mod(\Sigma,\tSigma)$ be the subgroup of the mapping class group $\Mod(\Sigma)$ that
lifts to $\tSigma$.
Consider some nonzero $\vec{v} \in \HHScc_1(\tSigma;\Q)$.
Our goal is to prove
the $\Mod(\Sigma,\tSigma)$-orbit of $\vec{v}$ is infinite.

Let $\sigma$ be the set of all topological types of nontrivial simple closed curves on
$\Sigma$.  Since $\HHScc_1(\tSigma;\Q) = \HH_1(\tSigma;\Q)$, Lemma \ref{lemma:twistfixed} implies that
$\vec{v}$ is not fixed by the group $D_{\sigma}$, so there exists
some nontrivial simple closed curve $\gamma$ on $\Sigma$ 
such that $\ttau_{\gamma}(\vec{v}) \neq \vec{v}$.  Some power of $\ttau_{\gamma}$ is the image
in the symplectic group of an element of $\Mod(\Sigma,\tSigma)$, so it is enough to prove
that the elements $\ttau_{\gamma}^n(\vec{v})$ as $n$ ranges over $\Z$ are all distinct.

Write $\pi^{-1}(\gamma)$ as a disjoint union of simple
closed curves on $\tSigma$:
\[\pi^{-1}(\gamma) = \tgamma_1 \sqcup \cdots \sqcup \tgamma_k.\]
There are then positive integers $e_1,\ldots,e_k$ such that
\[\ttau_{\gamma}(\vec{v}) = \vec{v} + \sum_{j=1}^k e_j \omega(\vec{v},[\tgamma_j]) \cdot [\tgamma_j].\]
Setting
\[\vec{w} = \sum_{j=1}^k e_j \omega(\vec{v},[\tgamma_j]) \cdot [\tgamma_j],\]
the fact that $\ttau_{\gamma}(\vec{v}) \neq \vec{v}$ implies that $\vec{w} \neq 0$.  For $n \in \Z$, we
have
$\ttau_{\gamma}^n(\vec{v}) = \vec{v} + n \vec{w}$.
Since $\vec{w} \neq 0$, the elements $\vec{v} + n \vec{w}$ as $n$ ranges over $\Z$ are all distinct, as desired.
\end{proof}

\section{The symplectic nature of simple closed curves}
\label{section:sccsymplectic}

This section contains the proof of Theorem \ref{maintheorem:sccsymplectic2} (which generalizes
Theorem \ref{maintheorem:sccsymplectic}).

\subsection{Notation}
\label{section:symplecticnotation}

The notation is similar to that of \S \ref{section:putmanwieland}: 
\begin{itemize}
\item $\pi\colon \tSigma \rightarrow \Sigma$ is a finite branched cover between closed oriented surfaces, and
$g$ is the genus of $\tSigma$.
\item $H = \HH_1(\tSigma;\Q) \cong \Q^{2g}$, and $\omega$ is the algebraic intersection form on $H$.
\item $\sigma$ is a set of topological types of nontrivial simple closed curves on $\Sigma$.
\item $D_{\sigma}$ is the subgroup of $\Sp(H,\omega) \cong \Sp_{2g}(\Q)$ generated by the elements
$\ttau_{\gamma}$ as $\gamma$ ranges over nontrivial simple closed curves on $\Sigma$ whose
topological type lies in $\sigma$.
\end{itemize}

\subsection{Symplectic criterion}

We will need the following criterion for a subspace of $H$ to be a symplectic subspace:

\begin{lemma}
\label{lemma:symplecticcriterion}
Let the notation be as above, and let $D$ be a subgroup of $\Sp(H,\omega)$.  Assume
that the action of $D$ on $H$ is semisimple.\footnote{That is, the $D$-representation $H$
decomposes as a direct sum of irreducible representations.}
Then\footnote{Here just like in Lemma \ref{lemma:twistfixed} the superscript indicates we are taking invariants.} $H^{D}$ is a symplectic subspace of $H$.
\end{lemma}

This is well-known; see, e.g., \cite[Lemme 4.14]{DeligneWeilII}.  For completeness, we include a proof.

\begin{proof}[Proof of Lemma \ref{lemma:symplecticcriterion}]
The symplectic form $\omega$ induces a $D$-equivariant
isomorphism 
\[\phi\colon H \xrightarrow{\cong} H^{\ast}.\]
We want to prove $\omega$ also induces an isomorphism between $H^D$ and $(H^D)^{\ast}$.  
Letting $\iota\colon H^D \hookrightarrow H$ be the inclusion and letting $\iota^{\ast}\colon H^{\ast} \rightarrow (H^D)^{\ast}$
be its dual,\footnote{This dual restricts a linear map $\lambda\colon H \rightarrow \Q$ to $H^D$.} 
our goal is equivalent to proving that the composition
\begin{equation}
\label{eqn:composition1}
H^D \xrightarrow{\iota} H \xrightarrow{\phi} H^{\ast} \xrightarrow{\iota^{\ast}} (H^D)^{\ast}
\end{equation}
is an isomorphism.

Since $\phi$ is a $D$-equivariant isomorphism, it restricts to an isomorphism on $D$-invariants, i.e., an
isomorphism
\[\phi^D\colon H^D \xrightarrow{\cong} \left(H^{\ast}\right)^D.\]
A linear map $\lambda\colon H \rightarrow \Q$ in $H^{\ast}$ is $D$-invariant if and only if it factors through the
$D$-coinvariants 
\[H_D = H / \SpanSet{$d \cdot h - h$}{$d \in D$ and $h \in H$}.\]  
We thus get an isomorphism
\[\mu\colon \left(H^{\ast}\right)^D \xrightarrow{\cong} \left(H_D\right)^{\ast}.\]
The projection $H \rightarrow H_D$ restricts to a map $\eta\colon H^D \rightarrow H_D$.  Since 
the action of $D$ is semisimple, the map $\eta$ is an isomorphism.\footnote{Indeed, 
if $H = H^D \oplus V_1 \oplus \cdots \oplus V_n$ with the $V_i$ nontrivial
irreducible representations of $D$, then
$H_D = (H^D)_D \oplus (V_1)_D \oplus \cdots \oplus (V_n)_D = H^D \oplus 0 \oplus \cdots \oplus 0 = H^D$.}
  Taking its dual, we get
an isomorphism
\[\eta^{\ast}\colon \left(H_D\right)^{\ast} \xrightarrow{\cong} \left(H^D\right)^{\ast}.\]
The composition
\begin{equation}
\label{eqn:composition2}
H^D \xrightarrow{\phi^D} \left(H^{\ast}\right)^D \xrightarrow{\mu} \left(H_D\right)^{\ast} \xrightarrow{\eta^{\ast}} \left(H^D\right)^{\ast}
\end{equation}
of isomorphisms is an isomorphism, and reflecting on the maps we see that the compositions \eqref{eqn:composition1} and
\eqref{eqn:composition2} are the same.  We conclude that \eqref{eqn:composition1} is an isomorphism, as desired.
\end{proof}

\subsection{Semisimplicity}

Our goal is to apply Lemma \ref{lemma:symplecticcriterion} to the group $D_{\sigma}$
from \S \ref{section:symplecticnotation}, which requires verifying the following:

\begin{lemma}
\label{lemma:dsemisimple}
Let the notation be as above.  Then the group $D_{\sigma}$ acts semisimply on $H$.
\end{lemma}
\begin{proof}
Let $\bD_{\sigma}$ be the Zariski closure of $D_{\sigma}$ in $\Sp(H,\omega)$.  It is enough
to prove that $\bD_{\sigma}$ acts semisimply on $H$.  For this, it is 
enough to prove that $\bD_{\sigma}$ is a semisimple algebraic group.\footnote{See
\cite{BorelBook, MilneBook} for textbook references on algebraic groups.  One
key property of semisimple algebraic groups over $\Q$ is that all of their finite-dimensional representations
are semisimple \cite[Proposition 22.41]{MilneBook}.}

Regard $\Sigma$ as a closed surface with marked points at the branch points
of $\pi\colon \tSigma \rightarrow \Sigma$.  To simplify things, if there are no branch points introduce a
single additional marked point on $\Sigma$, and regard its preimage in $\tSigma$ as a collection of branch
points of order $1$.  Let $\cM(\Sigma)$ be the moduli space of Riemann surfaces $S$ with marked
points such that $S \cong \Sigma$ as surfaces with marked points.  The (orbifold) fundamental group of $\cM(\Sigma)$ is
thus the mapping class group $\Mod(\Sigma)$.  

We can find a finite-index subgroup $\Gamma$ of $\Mod(\Sigma)$ such that each element
of $\Gamma$ can be lifted to a homeomorphism of $\tSigma$ fixing all the marked points.
Since there is at least one marked point, these lifts are unique up to homotopy, so $\Gamma$
acts on $\HH_1(\tSigma;\Q)$ in a well-defined way.
Shrinking $\Gamma$ if necessary, we can also assume that $\Gamma$ is torsion-free.
Let $\cM_{\Gamma}(\Sigma)$ be the cover of $\cM(\Sigma)$ corresponding to $\Gamma$.

Since $\Gamma$ is torsion-free, $\cM_{\Gamma}(\Sigma)$ is a fine moduli space.  It thus has a universal curve
$\cU \rightarrow \cM_{\Gamma}(\Sigma)$ whose fiber over $S \in \cM_{\Gamma}(\Sigma)$ is $S$.  Replacing
$\Gamma$ by a deeper finite-index subgroup if necessary, we can find a fiberwise branched cover
$\tcU \rightarrow \cM_{\Gamma}(\Sigma)$ of $\cU \rightarrow \cM_{\Gamma}(\Sigma)$ whose fibers are
the branched cover $\tSigma$ of $\Sigma$.

The monodromy representation of $\pi_1(\cM_{\Gamma}(\Sigma)) \cong \Gamma$ on $\HH_1$ of the fibers
is thus exactly the action of $\Gamma$ on $H=\HH_1(\tSigma;\Q)$ obtained by lifting mapping classes
through the branched cover $\tSigma \rightarrow \Sigma$.  The image of this representation lies
in $\Sp(H,\omega)$.  Let $\bG$ be the Zariski closure in $\Sp(H,\omega)$ of the image of $\Gamma$.
Deligne's semisimplicity theorem \cite[Corollaire 4.2.9]{DeligneHodgeII} implies that $\bG$ is a semisimple algebraic group.

From the definition \eqref{eqn:defined} of $d(\gamma)$ for nontrivial simple closed curves 
$\gamma$ on $\Sigma$,
it only achieves finitely many values (depending on the degree of the cover $\tSigma \rightarrow \Sigma$).
Pick some $m \geq 1$ such that the following two properties hold for each nontrivial 
simple closed curve $\gamma$ whose topological type lies in $\sigma$:
\begin{itemize}
\item $d(\gamma)$ divides $m$.
\item $T_{\gamma}^m \in \Gamma$.
\end{itemize}
Let $E$ be the subgroup of $\Mod(\Sigma)$ generated by all the $T_{\gamma}^m$ as
$\gamma$ ranges over nontrivial simple closed curves on $\Sigma$ whose topological type
lies in $\sigma$.  For such a $\gamma$, we have
\[f T_{\gamma}^m f^{-1} = T_{f(\gamma)}^m \quad \text{for all $f \in \Mod(\Sigma)$}.\]
It follows that $E$ is a normal
subgroup of $\Mod(\Sigma)$.  By construction, $E \subset \Gamma$.  For each nontrivial
simple closed curve $\gamma$ on $\Sigma$ whose topological type lies in $\sigma$,
recall that $\ttau_{\gamma}$ is the image of $T_{\gamma}^{d(\gamma)}$ in $\Sp(H,\omega)$.  The
Zariski closure in $\Sp(H,\omega)$ of the subgroup generated by $\ttau_{\gamma}$
is\footnote{The point here is that the subgroup generated by $\ttau_{\gamma}$ is the integer
points in the one-parameter subgroup $\ttau_{\gamma,t}$, and the Zariski closure of $\Z$ in $\Q$ is $\Q$.} 
the one-parameter subgroup $\ttau_{\gamma,t}$ defined by
\[\ttau_{\gamma,t}(h) = h + \sum_{j=1}^k t e_j \omega(h,[\tgamma_j]) \cdot [\tgamma_j] \quad \text{for $h \in H$ and $t \in \Q$}.\]
The group $\bD_{\sigma}$ is generated by these one-parameter subgroups.\footnote{This uses the fact that
the subgroup of an algebraic group generated by a set of algebraic subgroups is algebraic, i.e., Zariski closed \cite[Proposition 2.2]{BorelBook}.}
Since one-parameter subgroups are connected, we deduce that $\bD_{\sigma}$ is connected.

The Zariski closure in $\Sp(H,\omega)$ of the subgroup generated by 
\[T_{\gamma}^m = \left(T_{\gamma}^{d(\gamma)}\right)^{m/d(\gamma)}\] 
is the same one-parameter subgroup $\ttau_{\gamma,t}$.  It follows that
the Zariski closure of the image of $E$ in $\Sp(H,\omega)$ is also $\bD_{\sigma}$.
Since $E$ is a normal subgroup of $\Mod(\Sigma)$, it follows that
$\bD_{\sigma}$ is a normal subgroup of $\bG$.  Since $\bG$ is semisimple, so\footnote{Every connected normal
subgroup of a semisimple algebraic group is semisimple \cite[Theorem 21.51]{MilneBook}.}
is $\bD_{\sigma}$, as desired.
\end{proof}

\subsection{Symplectic subspace}

We now prove Theorem \ref{maintheorem:sccsymplectic2}.

\begin{proof}[Proof of Theorem \ref{maintheorem:sccsymplectic2}]
The statement we must prove is as follows.
Let $\pi\colon \tSigma \rightarrow \Sigma$ be a finite branched cover between
closed oriented surfaces 
and $\sigma$ be a set of topological types of nontrivial simple closed curves on $\Sigma$.
We must show that $H^{\sigma}=\HH_1^{\sigma}(\tSigma;\Q)$ is a symplectic subspace of 
$H=\HH_1(\tSigma;\Q)$,
or equivalently that $(H^{\sigma})^{\perp}$ is a symplectic subspace.
Lemma \ref{lemma:twistfixed} implies that
\[(H^{\sigma})^{\perp} = H^{D_{\sigma}},\]
and Lemma \ref{lemma:dsemisimple} says that the group $D_{\sigma}$ acts semisimply on $H$.
The result thus follows from Lemma \ref{lemma:symplecticcriterion}.
\end{proof}

\section{Pants homology}
\label{section:pants}

In this section, we prove Theorem \ref{maintheorem:pants2} (which implies
Theorem \ref{maintheorem:pants}).

\begin{proof}[Proof of Theorem \ref{maintheorem:pants2}]
We first recall the statement.
Let $\pi\colon \tSigma \rightarrow \Sigma$ be 
a finite branched cover between closed oriented surfaces
and $\cP$ be a pants decomposition of $\Sigma$.  Let $\sigma$ be the set of topological types of nontrivial curves
appearing in $\cP$.  We must prove that $H=\HH_1(\tSigma;\Q)$ is spanned by $H^{\sigma}=\HH_1^{\sigma}(\tSigma;\Q)$ and
the set of homology classes of cycles $\tgamma$ on $\tSigma$ such that $\pi(\tgamma)$ is disjoint
from all curves in $\cP$.

Theorem \ref{maintheorem:sccsymplectic2} says that $H^{\sigma}$ is a symplectic subspace of $H$, so
\[H = H^{\sigma} \oplus \left(H^{\sigma}\right)^{\perp}.\]
It is thus enough to prove that $(H^{\sigma})^{\perp}$ is spanned by
the homology classes of cycles $\tgamma$ on $\tSigma$ such that
$\pi(\tgamma)$ is disjoint from all the curves in $\cP$.  

Recall that we are working with homology with rational coefficients.  Every
element of $H$ (and hence $(H^{\sigma})^{\perp}$) is a multiple of an integral
class, and every integral class can be represented by an oriented multicurve.
Therefore, consider an oriented multicurve
$\tgamma$ on $\tSigma$ such that $[\tgamma] \in \left(H^{\sigma}\right)^{\perp}$.
It is enough to prove that
$\tgamma$ is homologous to an oriented multicurve $\tgamma'$ such that
$\pi(\tgamma)$ is disjoint from all the curves in $\cP$.

Our pants decomposition $\cP$ looks like the following:\\
\centerline{\psfig{file=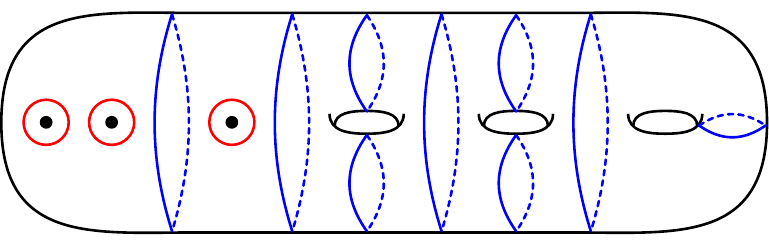,scale=1}}
Write $\cP = \{\delta_1,\ldots,\delta_n\}$.
Call the $\delta_j$ that bound disks containing marked points the {\em boundary loops} (red in the above figure)
and the other $\delta_j$ the {\em interior loops} (blue in the above figure).  
Enumerate the components of $\pi^{-1}(\delta_j)$ as $j$ ranges over $1 \leq j \leq n$ as
$\{\tdelta_1,\ldots,\tdelta_m\}$.  Call the $\tdelta_j$ that project to boundary loops the {\em lifted boundary loops}
and the $\tdelta_j$ that project to interior loops the {\em lifted interior loops}.

Put the oriented multicurve $\tgamma$ in general position with respect to the $\tdelta_j$.
The lifted boundary loops bound disks in $\tSigma$ containing a single branch point.  Isotope
$\tgamma$ such that it is disjoint from all these disks, and in particular is disjoint from
all the lifted boundary loops.

Let $\omega$ be the algebraic intersection form.  Consider a lifted interior loop $\tdelta_j$.  Since $[\tdelta_j] \in H^{\sigma}$ and $[\tgamma] \in (H^{\sigma})^{\perp}$, we have
$\omega([\tgamma],[\tdelta_j]) = 0$.  This implies that the number of positively oriented intersection points
of $\tgamma$ with $\tdelta_j$ is the same as the number of negatively oriented intersection points.
We can then modify
$\tgamma$ as follows to make it disjoint from $\tdelta_j$:\\
\centerline{\psfig{file=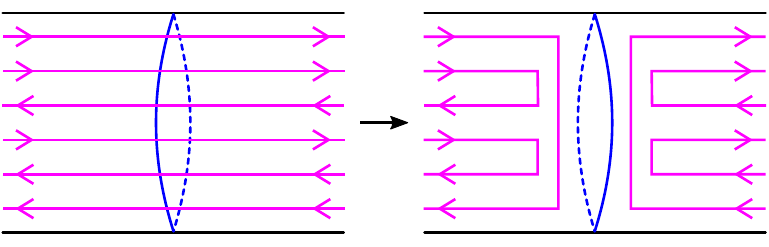,scale=1}}
The result is an oriented multicurve that is homologous to $\tgamma$.  Doing this for each
lifted interior loop, we obtain an oriented multicurve $\tgamma'$ such that $[\tgamma'] = [\tgamma]$
and such that $\tgamma'$ is disjoint from all the $\tdelta_j$ and does not lie in any of the disks
bounded by lifted boundary loops.  This implies $\pi(\tgamma)$ is disjoint from
all the curves in $\cP$, as desired.
\end{proof}

\section{A non-symplectic example}
\label{section:nonsymplectic}

This section contains the proof of Theorem \ref{maintheorem:nonsymplectic}, which asserts that for all closed oriented surfaces
$\Sigma$ of genus at least $2$, there exists
a finite unbranched cover $\pi\colon \tSigma \rightarrow \Sigma$ such that for
$\sigma$ the set of nonseparating simple closed curves on $\Sigma$, the subspace
$\HH_1^{\sigma}(\tSigma;\Z)$ is not a symplectic subspace of $\HH_1(\tSigma;\Z)$..

\subsection{Reduction}

We start with the following.

\begin{lemma}
\label{lemma:reduction}
Let $V$ be a finitely generated free abelian group equipped with a symplectic form and let $W$ be
a subgroup of $V$.  Assume that $W$ is a symplectic subspace of $V$ and that $W \otimes_{\Z} \Q = V \otimes_{\Z} \Q$.
Then $W = V$.
\end{lemma}
\begin{proof}
Since $W$ is a symplectic subspace of $V$, we have $V = W \oplus W^{\perp}$.  Since
\[W \otimes_{\Z} \Q = V \otimes_{\Z} \Q = \left(W \otimes_{\Z} \Q\right) \oplus \left(W^{\perp} \otimes_{\Z} \Q\right),\]
it follows that $W^{\perp} \otimes_{\Z} \Q = 0$.  We conclude that $W^{\perp} = 0$ and thus
that $W = V$.
\end{proof}

It is therefore enough to construct a finite unbranched cover $\pi\colon \tSigma \rightarrow \Sigma$
such that 
\[\HH_1^{\sigma}(\tSigma;\Q) = \HH_1(\tSigma;\Q) \quad \text{but} \quad \HH_1^{\sigma}(\tSigma;\Z) \neq \HH_1(\tSigma;\Z).\]
For $\ell \geq 2$, let $\pi\colon \Sigma[\ell] \rightarrow \Sigma$ be the cover corresponding
to the homomorphism
\[\pi_1(\Sigma) \longrightarrow \HH_1(\Sigma;\Z/\ell).\]
By the above, it is enough to prove the following theorem, which we will do in
the remainder of this section:

\begin{theorem}
\label{theorem:nonsymplectic}
Let $\Sigma$ be a closed oriented surface of genus at least $2$ and $\sigma$
be the set of nonseparating simple closed curves on $\Sigma$.  Fix some $\ell \geq 2$.
The following then hold:
\begin{itemize}
\item[(i)] We have $\HH_1^{\sigma}(\Sigma[\ell];\Q) = \HH_1(\Sigma[\ell];\Q)$.
\item[(ii)] If $\ell \geq 3$, then $\HH_1^{\sigma}(\Sigma[\ell];\Z) \neq \HH_1(\Sigma[\ell];\Z)$.
\end{itemize}
\end{theorem}

Part (ii) is a theorem of Irmer \cite[Lemma 6]{IrmerCovers}.  We will give
a simplified version of her argument below that avoids most of its complicated combinatorial group
theory.

\subsection{Rational equality}

We start by proving part (i) of Theorem \ref{theorem:nonsymplectic}.  In fact, we prove
a more general result:

\begin{proposition}
\label{proposition:sccabelian}
Let $\Sigma$ be a closed surface and $\sigma$ be the set of nonseparating simple
closed curves on $\Sigma$.  Let $\tSigma \rightarrow \Sigma$ be a finite unbranched abelian
cover.  Then $\HH_1^{\sigma}(\tSigma;\Q) = \HH_1(\tSigma;\Q)$.
\end{proposition}
\begin{proof}
It is enough to prove that $\HH_1^{\sigma}(\tSigma;\Q)^{\perp} = 0$.
During the proof of Theorem \ref{maintheorem:sccsymplectic2}, we showed that
\[\HH_1^{\sigma}(\tSigma;\Q)^{\perp}=\HH_1(\tSigma;\Q)^{D_{\sigma}}.\]
It is thus enough to show that the group $D_{\sigma}$ fixes no nonzero vectors in $\HH_1(\tSigma;\Q)$.
In fact, the $D_{\sigma}$-orbits of all nonzero vectors in $\HH_1(\tSigma;\Q)$
are infinite.  This follows from work of Looijenga \cite{LooijengaPrym}, but since
it is only implicit in \cite{LooijengaPrym} we give a complete proof.

Let $G$ be the deck group of the finite abelian cover $\tSigma \rightarrow \Sigma$.  The actions
of $G$ and $D_{\sigma}$ on $\HH_1(\tSigma;\Q)$ commute, so the action of $D_{\sigma}$ preserves
the decomposition of $\HH_1(\tSigma;\Q)$ into $G$-isotypic components.  Let
$V$ be an irreducible representation of $G$ over $\Q$ and let $W$ be
the $V$-isotypic component of $\HH_1(\tSigma;\Q)$.  We must prove
that the $D_{\sigma}$-orbits of all nonzero vectors in $W$ are infinite.

Since $V$ is an irreducible representation of the finite abelian group $G$, there is a finite
cyclic quotient $\phi\colon G \rightarrow \Z/d$ such that\footnote{Here is a sketch of this standard
fact.  Since $G$ is abelian, the action of $G$ on $V$ comes from a homomorphism $\iota\colon G \rightarrow \End_G(V)$.  Since
$V$ is irreducible, Schur's Lemma says that $\End_G(V)$ is a division algebra over $\Q$.  Let $F$ be the $\Q$-subalgebra of
$\End_G(V)$ generated by $\text{Im}(\iota)$.  It is an easy exercise to show that for $f \in \End_G(V)$ nonzero, $f^{-1}$ can
be expressed as a polynomial in $f$.  It follows that $F$ is closed under taking multiplicative inverses.
Since $G$ is abelian, this implies that $F$ is a commutative division ring, i.e., a field.
The result now follows from the fact that a finite subgroup of $F^{\times}$ like the image of $\iota\colon G \rightarrow \End_G(V)$
must be cyclic.}
action of $G$ on $V$ factors through $\phi$.
Let
\[\Sigma[\phi] = \tSigma / \ker(\phi),\]
and let $\pi\colon \Sigma[\phi] \rightarrow \Sigma$ be the projection, so $\pi\colon \Sigma[\phi] \rightarrow \Sigma$ is the degree-$d$ cyclic cover corresponding to $\ker(\phi)$.  
The group $\ker(\phi)$ acts trivially on $V$ and hence on $W$, so $W$ is a subrepresentation of\footnote{Here the
subscript indicates that we are taking the $\ker(\phi)$-coinvariants and the first isomorphism follows
from the transfer map.}
\[\HH_1(\tSigma;\Q)_{\ker(\phi)} \cong \HH_1(\tSigma/\ker(\phi);\Q) \cong \HH_1(\Sigma[\phi];\Q).\]
Letting $v \in \HH_1(\Sigma[\phi];\Q)$ be nonzero, it is thus enough to prove that the $D_{\sigma}$-orbit
of $v$ is infinite.

Let $\omega$ be the algebraic intersection form on $\HH_1(\Sigma;\Q)$.  There is a surjection $\HH_1(\Sigma;\Z) \rightarrow G$.
Pick a surjection $\tphi\colon \HH_1(\Sigma;\Z) \rightarrow \Z$ making the
diagram
\begin{center}
\begin{tikzcd}
\HH_1(\Sigma;\Z) \arrow[two heads]{r}{\tphi} \arrow[two heads]{d} & \Z \arrow[two heads]{d} \\
G                \arrow[two heads]{r}{\phi}                       & \Z/d
\end{tikzcd}
\end{center}
commute.  Since $\omega$ is a symplectic form on $\HH_1(\Sigma;\Z)$, there
exists some $a_1 \in \HH_1(\Sigma;\Z)$ such that $\tphi(x) = \omega(a_1,x)$ for all $x \in \HH_1(\Sigma;\Z)$.
Since $\tphi$ is surjective, $a_1$ is primitive\footnote{That is, not divisible by any integers
greater than $1$.} and thus there exists some oriented nonseparating simple closed curve $\alpha_1$
such that $[\alpha_1] = a_1$.  Let $\beta_1$ and $S$ be as follows:\\
\centerline{\psfig{file=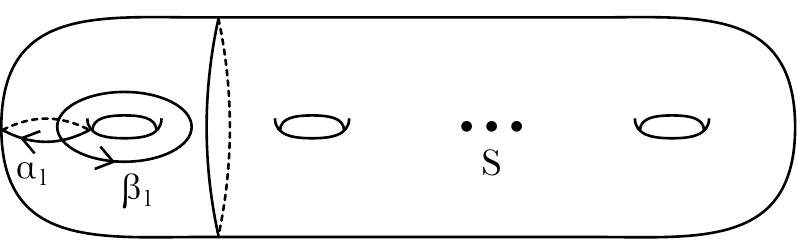,scale=1}}
The case $d=3$ of our cover $\pi\colon \Sigma[\phi] \rightarrow \Sigma$ is then
as follows:\\
\centerline{\psfig{file=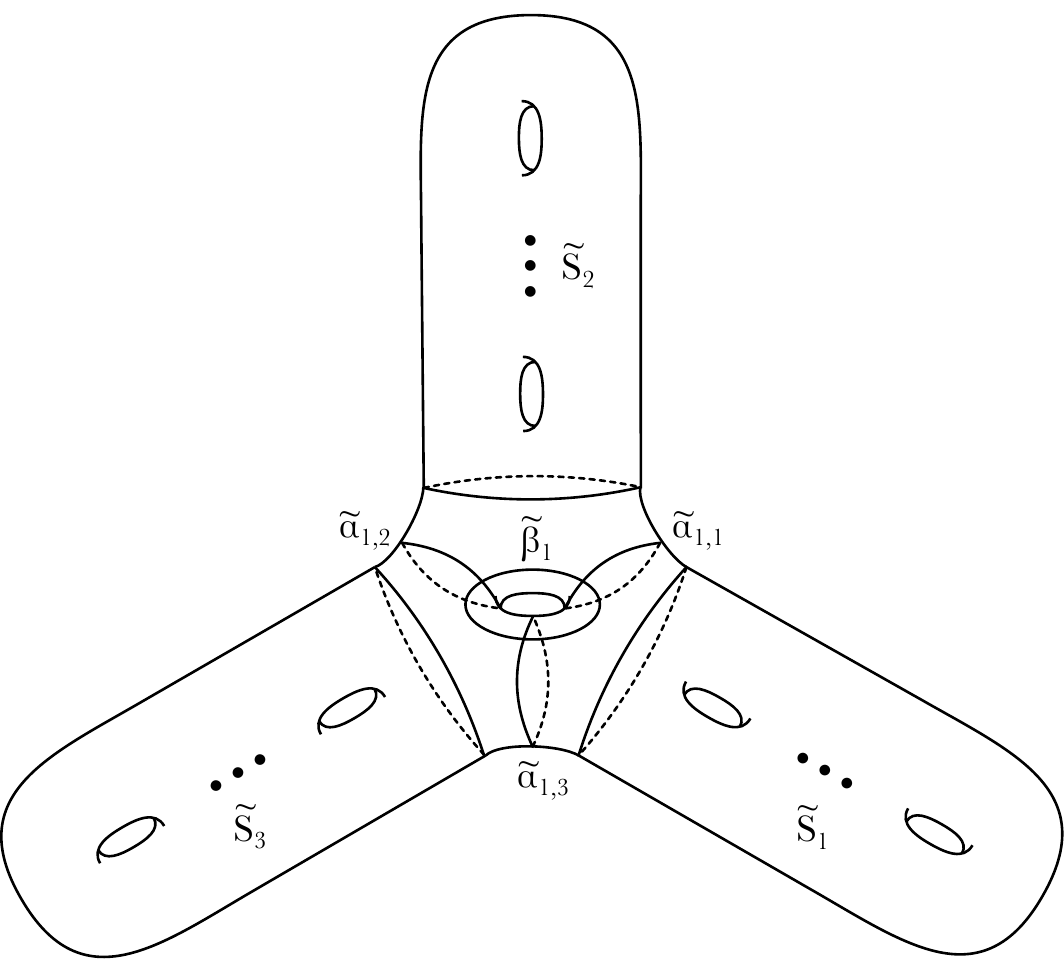,scale=1}}
More generally, we have
\begin{itemize}
\item $\pi^{-1}(S) = \tS_1 \sqcup \cdots \sqcup \tS_d$ with each $\tS_i$ projecting
homeomorphically to $S$; and
\item $\pi^{-1}(\beta_1) = \tbeta_1$, where $\tbeta_1$ is a simple closed curve
that $d$-fold covers $\beta_1$; and
\item $\pi^{-1}(\alpha_1) = \talpha_{1,1} \sqcup \cdots \sqcup \talpha_{1,d}$, where
$\talpha_{1,i}$ is a simple closed curve projecting homeomorphically to $\alpha_1$.
\end{itemize}
The curves $\talpha_{1,i}$ are all homologous, and we have
\[\HH_1(\Sigma[\phi];\Q) = \Span{[\talpha_{1,1}],[\tbeta_1]} \oplus \bigoplus_{i=1}^d \HH_1(\tS_i;\Q).\]
Recall that we are trying to prove that the nonzero $v \in \HH_1(\Sigma[\phi];\Q)$ has
an infinite $D_{\sigma}$-orbit.  In fact, we will find some $\ttau \in D_{\sigma}$ such that
the elements $\Set{$\ttau^n(v)$}{$n \geq 1$}$ are all
distinct.
Write
\[v = \lambda [\talpha_{1,1}] + \nu [\tbeta_1] + \sum_{i=1}^d v_i \quad \text{with $\lambda,\nu \in \Q$ and $v_i \in \HH_1(\tS_i;\Q)$.}\]
There are three cases.

The first is $\lambda \neq 0$.  In this case, $T_{\beta_1}^d \in \Mod(\Sigma)$ lifts to
$T_{\tbeta_1} \in \Mod(\Sigma[\phi])$.  We have
\[T_{\tbeta_1}^n(v) = \lambda [\talpha_{1,1}] + (\nu - n \lambda) [\tbeta_1] + \sum_{i=1}^d v_i \quad \text{for $n \geq 1$}.\]
These are all distinct elements.  Since $D_{\sigma}$ is defined via the lifts to the cover\footnote{Remember
that the degree-$d$ cyclic cover $\Sigma[\phi] \rightarrow \Sigma$ is a subcover of $\tSigma \rightarrow \Sigma$, i.e., the
map $\tSigma \rightarrow \Sigma$ factors as $\tSigma \rightarrow \Sigma[\phi] \rightarrow \Sigma$.}
$\tSigma \rightarrow \Sigma$, there is some $\ell \geq 1$ (necessarily divisible by $d$) such that 
the element $\ttau_{\beta_1} \in D_{\sigma}$
is induced by the lift of $T_{\beta_1}^{\ell} = (T_{\beta_1}^d)^{\ell/d}$.  We conclude that
the elements $\Set{$\ttau_{\beta_1}^n(v)$}{$n \geq 1$}$ are all
distinct.

The second is $\nu \neq 0$.  In this case, $T_{\alpha_1} \in \Mod(\Sigma)$ lifts to
$T_{\talpha_{1,1}} \cdots T_{\talpha_{1,d}} \in \Mod(\Sigma[\phi])$.  Since
the $\talpha_{1,i}$ are all homologous, we have
\[(T_{\talpha_{1,1}} \cdots T_{\talpha_{1,d}})^n(v) = (\lambda + d n \nu) [\talpha_{1,1}] + \nu [\tbeta_1] + \sum_{i=1}^d v_i \quad \text{for $n \geq 1$}.\]
These are all distinct elements.  Just like in the previous case, we conclude that
the elements $\Set{$\ttau_{\alpha_1}^n(v)$}{$n \geq 1$}$ are all
distinct.

The third is that some $v_i$ is nonzero.  Reordering, assume that $v_1 \neq 0$.  Let
$\ov_1 \in \HH_1(S;\Q)$ be the image of $v_1 \in \HH_1(\tS_1;\Q)$.
Pick an oriented simple closed curve $\gamma$ on $S$ with $\omega([\gamma],\ov_1)$
nonzero.  Let $\tgamma_1 \sqcup \cdots \sqcup \tgamma_d$ be the preimage
of $\gamma$ in $\Sigma[\phi]$, ordered such that $\tgamma_i \in \tS_i$.
By construction,
$T_{\gamma} \in \Mod(\Sigma)$ lifts to $T_{\tgamma_1} \cdots T_{\tgamma_d} \in \Mod(\Sigma[\phi])$. 
We have
\[(T_{\tgamma_1} \cdots T_{\tgamma_d})^n(v) = \lambda [\talpha_{1,1}] + \nu [\tbeta_1] + \sum_{i=1}^d (v_i + n \omega([\tgamma_i],v_i) [\tgamma_i]) \quad \text{for $n \geq 1$}.\] 
Since $\omega([\tgamma_1],v_1) = \omega([\gamma],\ov_1) \neq 0$, these are all distinct.  Just like before,
we conclude that the elements $\Set{$\ttau_{\gamma}^n(v)$}{$n \geq 1$}$ are all distinct.
\end{proof}

\subsection{Nilpotent preliminaries}

Before we can prove part (ii) of Theorem \ref{theorem:nonsymplectic}, we need
some preliminary results.  Let $F_n$ be the free group on $\{x_1,\ldots,x_n\}$.  Fix
some $\ell \geq 3$.  Define\footnote{When reading this for the first time, it might be easier
to assume that $\ell$ is odd, so $\hell = \ell$.}
\[\hell = \begin{cases}
\ell & \text{if $\ell$ is odd},\\
\ell/2 & \text{if $\ell$ is even}.
\end{cases}\]
Since $\ell \geq 3$, we have $\hell \geq 2$.  Define $N_n[\ell]$ to be the quotient of $F_n$ by
the normal subgroup generated by the following elements:\footnote{Here ``N'' stands
for ``nilpotent''.}
\begin{itemize}
\item The third term $[F_n,[F_n,F_n]]$ of the lower central series.
\item The subgroup $[F_n,F_n^{\times \hell}]$, i.e., the subgroup generated
by commutators $[u,v^{\hell}]$ as $u$ and $v$ range over elements of $F_n$.
\end{itemize}
We will use boldface letters to denote elements of $N_n[\ell]$, and in particular
will let $\{\bx_1,\ldots,\bx_n\}$ be the generators of $N_n[\ell]$ coming
from the generators $\{x_1,\ldots,x_n\}$ for $F_n$.  The abelianization of $N_n[\ell]$
is $\Z^n$, and for $\bu \in N_n[\ell]$ we will write $\obu \in \Z^n$ for its image
in the abelianization and $\hobu \in (\Z/\hell)^n$ for the image of $\obu$ under the mod-$\hell$ reduction map.

The following lemma clarifies the nature of $N_n[\ell]$:

\begin{lemma}
\label{lemma:nilpotent}
For $n \geq 2$ and $\ell \geq 3$, we have a central extension
\[1 \longrightarrow \wedge^2 (\Z/\hell)^n \longrightarrow N_n[\ell] \longrightarrow \Z^n \longrightarrow 1.\]
Here the map $N_n[\ell] \rightarrow \Z^n$ is the abelianization map taking $\bu \in N_n[\ell]$ to
$\obu \in \Z^n$, and for $\bu,\bv \in N_n[\ell]$ the commutator $[\bu,\bv] \in N_n[\ell]$ is the central element
$\hobu \wedge \hobv \in \wedge^2 (\Z/\hell)^n$.
\end{lemma}
\begin{proof}
It is immediate from Magnus--Witt's work on the lower central series of a free group (\cite{Magnus, Witt}; see
\cite{SerreLie} for a textbook account) that 
\[\frac{[F_n,F_n]}{[F_n,[F_n,F_n]]} \cong \wedge^2 \Z^n,\]
with $[u,v] \in [F_n,F_n]$ mapping to $\ou \wedge \ov \in \wedge^2 \Z^n$.  Here $\ou,\ov \in \Z^n$ are the
images of $u,v \in F_n$ in its abelianization.  This fits into a
central extension
\[1 \longrightarrow \wedge^2 \Z^n \longrightarrow \frac{F_n}{[F_n,[F_n,F_n]]} \longrightarrow \Z^n \longrightarrow 1.\]
To get $N_n[\ell]$ from the middle group in this extension, one quotients out the image of
$[F_n,F_n^{\times \hell}]$, which maps to the kernel of the map
\[\wedge^2 \Z^n \longrightarrow \wedge^2 (\Z/\hell)^n.\]
The lemma follows.
\end{proof}

In the rest of this section, we will identify $\wedge^2 (\Z/\hell)^n$ with the corresponding central subgroup
of $N_n[\ell]$.  The following calculation lies at the heart of our arguments:

\begin{lemma}
\label{lemma:powerlemma}
For $n \geq 2$ and $\ell \geq 3$, we have
$(\bu \bv)^{\ell} = \bu^{\ell} \bv^{\ell}$ for all $\bu,\bv \in N_n[\ell]$.
\end{lemma}
\begin{proof}
To transform $(\bu \bv)^{\ell}$ into $\bu^{\ell} \bv^{\ell}$, we must commute each $\bu$ past all the
$\bv$ terms to its left.  Each time we commute a $\bu$ past a $\bv$, we must introduce a commutator
$[\bv,\bu] = \hobv \wedge \hobu$.  This commutator is central, so it can moved all the way to the right.
The first $\bu$ must be commuted with $0$ copies of $\bv$, the second with $1$ copy of $\bv$, the third with $2$ copies of $\bv$, etc.
In the end, we see that
\[(\bu \bv)^{\ell} = \bu^{\ell} \bv^{\ell} [\bv,\bu]^{0+1+2+\cdots+(\ell-1)} = \bu^{\ell} \bv^{\ell} [\bv,\bu]^{\ell(\ell-1)/2}.\]
Whether $\ell$ is even or odd,\footnote{The purpose of using $\hell$ is to ensure this.} the integer
$\ell(\ell-1)/2$ is divisible by $\hell$.  Since $[\bv,\bu] \in \wedge^2 (\Z/\hell)^n$, this implies
that $[\bv,\bu]^{\ell(\ell-1)/2} = 1$.  The lemma follows.
\end{proof}

Define $P_n[\ell]$ to be\footnote{Here ``P'' stands for ``power subgroup''.} the subgroup of $N_n[\ell]$ generated by 
$\Set{$\bu^\ell$}{$\bu \in N_n[\ell]$}$
and define $A_n[\ell]$ to be the subgroup\footnote{Here ``A'' stands for ``abelian subgroup''; see Lemma \ref{lemma:powersubgroup}.} of $N_n[\ell]$ 
generated by $P_n[\ell]$ and $\wedge^2 (\Z/\hell)^n$.  We then have:

\begin{lemma}
\label{lemma:powersubgroup}
For $n \geq 2$ and $\ell \geq 3$, the subgroup $P_n[\ell]$ is a central subgroup of $N_n[\ell]$ with $P_n[\ell] \cong \Z^n$, and
$A_n[\ell] = P_n[\ell] \times \wedge^2 (\Z/\hell)^n$.
\end{lemma}
\begin{proof}
The fact that $P_n[\ell]$ is a central subgroup follows from the fact that
\[[\bu^{\ell},\bv] = \hobu^{\ell} \wedge \hobv = \ell \left(\hobu \wedge \hobv\right) = 0 \quad \text{for all $\bu,\bv \in N_n[\ell]$}.\]
Recall that $N_n[\ell]$ is generated by the elements $\bx_1,\ldots,\bx_n$, which map to a basis for the abelianization
$\Z^n$.  The elements $\bx_i^{\ell} \in N_n[\ell]$ are central and map to linearly independent elements in the abelianization, so
\[P'_n[\ell] = \Set{$\bx_1^{\ell k_1} \cdots \bx_n^{\ell k_n}$}{$k_1,\ldots,k_n \in \Z$}\]
is a central subgroup satisfying $P'_n[\ell] \cong \Z^n$.  Moreover, letting $A'_n[\ell]$ be the subgroup of $N_n[\ell]$ generated
by $P'_n[\ell]$ and $\wedge^2 (\Z/\hell)^n$, we clearly have $A'_n[\ell] = P'_n[\ell] \times \wedge^2 (\Z/\hell)^n$.

To prove the lemma, it is therefore enough to prove that $P_n[\ell] = P'_n[\ell]$.  Since $\bx_i^{\ell} \in P_n[\ell]$ for all
$1 \leq i \leq n$, we have $P'_n[\ell] \subset P_n[\ell]$.  For the reverse inclusion, consider some $\bu \in N_n[\ell]$.
We must prove that $\bu^{\ell} \in P'_n[\ell]$.  We can find $k_1,\ldots,k_n \in \Z$ and
\[\bc \in [P_n[\ell],P_n[\ell]] = \wedge^2 (\Z/\hell)^n\]
such that
$\bu = \bx_1^{k_1} \cdots \bx_n^{k_n} \bc$.
Applying Lemma \ref{lemma:powerlemma} repeatedly, we deduce that
\[\bu^{\ell} = \bx_1^{\ell k_1} \cdots \bx_n^{\ell k_n} \bc^{\ell} = \bx_1^{\ell k_1} \cdots \bx_n^{\ell k_n} \in P'_n[\ell].\qedhere\]
\end{proof}

\subsection{Integral inequality}

We now prove part (ii) of Theorem \ref{theorem:nonsymplectic}

\begin{proof}[Proof of Theorem \ref{theorem:nonsymplectic}, part (ii)]
We first recall the statement.  Let $\Sigma$ be a closed oriented surface of genus $g \geq 2$ 
and $\sigma$ be the set of nonseparating simple closed curves on $\Sigma$.  Fix some $\ell \geq 3$, and
as above let
\[\hell = \begin{cases}
\ell & \text{if $\ell$ is odd},\\
\ell/2 & \text{if $\ell$ is even}.
\end{cases}\]
Since $\ell \geq 3$, we have $\hell \geq 2$.
We must prove that $\HH_1^{\sigma}(\Sigma[\ell];\Z) \neq \HH_1(\Sigma[\ell];\Z)$.

Recall that $\Sigma[\ell]$ is the cover corresponding to the homomorphism
\[\pi_1(\Sigma) \rightarrow \HH_1(\Sigma;\Z/\ell) \cong (\Z/\ell)^{2g}.\]
It follows that $\pi_1(\Sigma[\ell])$ is the kernel of this map, so $\pi_1(\Sigma[\ell])$ is the subgroup
of $\pi_1(\Sigma)$ generated by the following two subgroups:
\begin{itemize}
\item The commutator subgroup $[\pi_1(\Sigma),\pi_1(\Sigma)]$.
\item The subgroup $P$ generated by $\Set{$x^{\ell}$}{$x \in \pi_1(\Sigma)$}$.
\end{itemize}
Each nonseparating simple closed curve $x \in \pi_1(\Sigma)$ maps to a primitive\footnote{That is, not divisible by any integers except $\pm 1$.} element of $\HH_1(\Sigma;\Z)$, so the minimal power of $x$ that lies in $\pi_1(\Sigma[\ell])$ is $x^{\ell}$.  It follows
that the image $\oP$ of
$P$ in $\HH_1(\Sigma[\ell];\Z)$ contains $\HH_1^{\sigma}(\Sigma[\ell];\Z)$.  It is enough therefore
to prove that $\oP \neq \HH_1(\Sigma[\ell];\Z)$.

Let $\{a_1,b_1,\ldots,a_g,b_g\}$ be the standard generating set for $\pi_1(\Sigma)$ satisfying
the surface relation $[a_1,b_1]\cdots[a_g,b_g]=1$.  We can then define a homomorphism
$\phi\colon \pi_1(\Sigma) \rightarrow N_g[\ell]$ via the formulas
\[\phi(a_i) = \bx_i \quad \text{and} \quad \phi(b_i) = 1 \quad \text{for $1 \leq i \leq g$}.\]
The map $\phi$ takes $[\pi_1(\Sigma),\pi_1(\Sigma)]$ to the central subgroup
$\wedge^2 (\Z/\hell)^g$ and $P$ to the central subgroup $P_g[\ell]$ (see Lemma \ref{lemma:powersubgroup}).  It follows that $\phi$
takes $\pi_1(\Sigma[\ell])$ surjectively onto the abelian subgroup 
$A_g[\ell] = P_g[\ell] \times \wedge^2 (\Z/\hell)^g$
identified by Lemma \ref{lemma:powersubgroup}.
The restriction of $\phi$ to $\pi_1(\Sigma[\ell])$ thus factors through $\HH_1(\Sigma[\ell];\Z)$, and takes
$\oP \subset \HH_1(\Sigma[\ell];\Z)$ to the proper subgroup $P_g[\ell]$ of $A_g[\ell]$.  The theorem follows.
\end{proof}


\begin{thebibliography}{99}

\bibitem{BorelBook}
A. Borel,
{\it Linear algebraic groups},
Second edition. Graduate Texts in Mathematics, 126. Springer-Verlag, New York, 1991

\bibitem{DeligneHodgeII}
P. Deligne, Th\'{e}orie de Hodge. II, Inst. Hautes \'{E}tudes Sci. Publ. Math. No. 40 (1971), 5--57.

\bibitem{DeligneWeilII}
P. Deligne, La conjecture de Weil. II, Inst. Hautes \'{E}tudes Sci. Publ. Math. No. 52 (1980), 137--252.

\bibitem{Primer}
B. Farb\ and\ D. Margalit, {\it A primer on mapping class groups}, Princeton Mathematical Series, 49, Princeton Univ. Press, Princeton, NJ, 2012.

\bibitem{GrunewaldEtAl}
F. Grunewald, M. Larsen, A. Lubotzky, and J. Malestein, Arithmetic quotients of the mapping class group, Geom. Funct. Anal. 25 (2015), no.~5, 1493--1542. \arxiv{1307.2593}

\bibitem{IrmerCovers}
I. Irmer, Lifts of simple curves in finite regular coverings of closed surfaces, Geom. Dedicata 200 (2019), 67--76. \arxiv{1508.04815}

\bibitem{IvanovConjecture}
N.~V. Ivanov, Fifteen problems about the mapping class groups, in {\it Problems on mapping class groups and related topics}, 71--80, Proc. Sympos. Pure Math., 74, Amer. Math. Soc., Providence, RI. \arxiv{math/0608325}

\bibitem{KentMO}
A. Kent (https://mathoverflow.net/users/1335/autumn-kent), Homology generated by lifts of simple curves, URL (version: 2012-01-29): https://mathoverflow.net/q/86938

\bibitem{Klukowski}
A. Klukowski, Simple closed curves, non-kernel homology and Magnus embedding, preprint 2023. \arxiv{2304.13196}

\bibitem{KoberdaSantharoubane}
T. Koberda\ and\ R. Santharoubane, Quotients of surface groups and homology of finite covers via quantum representations, Invent. Math. 206 (2016), no.~2, 269--292. \arxiv{1510.00677}

\bibitem{LandesmanLitt1}
A. Landesman\ and\ D. Litt, Canonical representations of surface groups, preprint 2022. \arxiv{2205.15352}

\bibitem{LandesmanLitt2}
A. Landesman\ and\ D. Litt, Applications of the algebraic geometry of the Putman-Wieland conjecture, preprint 2022. \arxiv{2209.00718}

\bibitem{LandesmanLitt3}
A. Landesman\ and\ D. Litt, An introduction to the algebraic geometry of the Putman-Wieland conjecture, preprint 2022. \arxiv{2209.00717}

\bibitem{LooijengaPrym}
E.~J.~N. Looijenga, Prym representations of mapping class groups, Geom. Dedicata 64 (1997), no.~1, 69--83. 

\bibitem{MalesteinPutman}
J. Malestein\ and\ A. Putman, Simple closed curves, finite covers of surfaces, and power subgroups of ${\rm Out}(F_n)$, Duke Math. J. 168 (2019), no.~14, 2701--2726. \arxiv{1708.06486}

\bibitem{Magnus}
W. Magnus, \"{U}ber Beziehungen zwischen h\"{o}heren Kommutatoren, J. Reine Angew. Math. 177 (1937), 105--115.

\bibitem{Markovic}
V. Markovi\'{c}, Unramified correspondences and virtual properties of mapping class groups, Bull. Lond. Math. Soc. 54 (2022), no.~6, 2324--2337.

\bibitem{MarkovicTosic}
V. Markovi\'{c}\ and\ O. To\v{s}i\'{c}, The second variation of the Hodge norm and the Higher Prym Representations, preprint 2022.

\bibitem{MilneBook}
J. S. Milne,
{\it Algebraic groups. The theory of group schemes of finite type over a field}, 
Cambridge Studies in Advanced Mathematics, 170. Cambridge University Press, Cambridge, 2017. 

\bibitem{PutmanWieland}
A. Putman and B. Wieland, Abelian quotients of subgroups of the mappings class group and higher Prym representations, J. Lond. Math. Soc. (2) 88 (2013), no. 1, 79–96. \arxiv{1106.2747}

\bibitem{SerreLie}
J.-P. Serre, {\it Lie algebras and Lie groups}, corrected fifth printing of the second (1992) edition, Lecture Notes in Mathematics, 1500, Springer, Berlin, 2006. 

\bibitem{Witt}
E. Witt, Treue Darstellung Liescher Ringe, J. Reine Angew. Math. 177 (1937), 152--160.

\end{thebibliography}
\end{document}